\documentclass[a4paper,12pt]{amsart}

\usepackage[latin1]{inputenc}
\usepackage[T1]{fontenc}
\usepackage[english]{babel}

\usepackage{mathrsfs}
\usepackage{amscd}
\usepackage{amsfonts}
\usepackage{amsmath}
\usepackage{amssymb}
\usepackage{amstext}
\usepackage{amsthm}
\usepackage{amsbsy}
\usepackage{color}

\usepackage{xspace}
\usepackage[all]{xy}
\usepackage{graphicx}
\usepackage{url}
\usepackage{latexsym}

\makeatletter
\newcommand*{\rom}[1]{\expandafter\@slowromancap\romannumeral #1@}
\makeatother

\theoremstyle{definition}

\newcommand{\N}{\ensuremath{\mathbb{N}}}

\newtheorem{fact}{fact}
\newtheorem{example}[fact]{Example}
\newtheorem{thm}[fact]{Theorem}

\newtheorem{corollary}[fact]{Corollary}
\newtheorem{defini}[fact]{Definition}
\newtheorem{rem}[fact]{Remark}

\title{On the value group of a model of Peano Arithmetic}
\author{Merlin Carl}
\address{FB Mathematik \& Statistik, Universit\"at Konstanz,
Germany}
 \email{merlin.carl@uni-konstanz.de}
\author{Paola D'Aquino}
\address{Dipartimento di Matematica, Seconda Universit\`a di Napoli, Italy}
\email{paola.daquino@unina2.it}
\author{Salma Kuhlmann}
\address{FB Mathematik \& Statistik, Universit\"at Konstanz,
Germany}
 \email{salma.kuhlmann@uni-konstanz.de}
\date{\today}

\begin{document}
\begin{abstract}
We investigate $IPA$ - real closed fields, that is, real closed fields
which admit an integer part whose non-negative cone is a model of
Peano Arithmetic. We show that the value group of an $IPA$ - real closed
field is an exponential group in the residue field, and that the
converse fails in general. As an application, we classify (up to
isomorphism) value groups of countable recursively saturated
exponential real closed fields. We exploit this characterization to
construct countable exponential real closed fields which are not
$IPA$ - real closed fields.
\end{abstract}

\maketitle

\markright{ON THE VALUE GROUP OF A MODEL OF PEANO ARITHMETIC}

\section{Introduction}
\noindent
A real closed field ($RCF$)  is a model of $Th(\mathbb{R})$, the first order theory of  $\mathbb{R}$ in the  language of  ordered rings  $\mathcal L=\{ +,\cdot,0,1,<\}$. An integer part of a real closed field $R$ is a subring of $R$ that generalizes in a natural way the relation between the integers and the reals.
\begin{defini}

An \emph{integer part} ($IP$) for an ordered  field $R$ is a discretely ordered subring $Z$ such that for each $r\in R$, there exists some $z\in Z$ with $z\leq r < z+1$. This $z$ will be denoted by $\lfloor r\rfloor$.

\end{defini}

If $R$ is an ordered field, then $R$ is Archimedean if and only if
$\mathbb{Z}$ is (the unique) integer part.  We will consider only
non-Archimedean  real closed fields. Integer parts provide an
approach for building connections between real algebra and  models
of fragments of Peano Arithmetic ($PA$). Shepherdson
\cite{shepherdson} showed that  a discretely ordered ring $Z$ is an
integer part for some real closed field if and only if its
non-negative cone, i.e. $Z^{\geq 0}$ is a model of  $IOpen$ (see
below). Every real closed field has an integer part, see \cite{mr}.
However, an  integer part for $R$ need not be  unique, not even up
to elementary equivalence.
 %As shown in [MR] every real closed field has an integer part.
 Indeed it is difficult to provide algebraic criteria for integer parts to satisfy
properties beyond $IOpen$ (such as cofinality of primes, normality), especially
if they are to be constructed as recursive structures see \cite{Moniri}, and \cite{BKK} for an extensive discussion.

%{\color{red}Moreover, the general experience is that integer parts having
%properties of $\mathbb{N}$ beyond $IOpen$ (such as being normal or having cofinally many primes) are much harder to obtain: In particular,
%it is currently open whether each $RCF$ has a normal $IP$. For more information on obtaining recursive normal integer parts, see \cite{BO}.}

\smallskip

We recall that $PA$  is the first order theory in the
language $\mathcal L=\{ +,\cdot,0,1,<\}$  of discretely ordered
commutative rings with $1$ of which non-negative cone satisfies  the
induction axiom for each  first order formula. Fragments of $PA$ are
obtained by restricting the induction scheme to certain classes of
formulas. The following subsystems will be relevant for this paper:
\begin{enumerate}
 \item $IOpen$ (open induction) denotes induction restricted to open formulas, i.e. formulas that do not contain quantifiers;
 \item $I\Delta_{0}$ (bounded induction) denotes induction restricted to bounded formulas, i.e. formulas in which all quantifiers are bounded by terms of the language;
 \item $I\Delta_{0}+EXP$ is $I\Delta_{0}$ with an extra axiom stating that the exponential function is total. It is known that $I\Delta_{0}+EXP$ is a proper extension
of $I\Delta_{0}$, for details  see  \cite[Chapter V.3]{hp}.
\end{enumerate}

\smallskip

\noindent

In what follows we will not make any distinction between  the discrete ordered ring and its non-negative cone as a model of $PA$ or one of its fragments. For convenience we introduce the following definition.

%Recall that an ordered field $K$ is said to have left exponentiation iff there is an isomorphism from a group complement $A$ of $R_{v}$ in $(K,+,0,<)$ onto a group complement $B$ of $U_{v}^{>0}$ in $(K^{>0},\cdot,1,<)$.\\

\begin{defini}
%An $IPA$ is an integer part of a non-Archimedean real closed field
%which is a model of $PA$.
A non-Archimedean real closed field which
admits an integer part which is a model of $PA$ is
called an $IPA$ -  real closed field, or $IPA - RCF$. Any such $IP$ for the $IPA - RCF$ at hand is called an $IPA$ for it.
\end{defini}

%\begin{defini}
%A structure $\M$ is recursively saturated if for all recursive types $\tau (u,\overline{v})$ and for all $\overline{a}\in M$ if $\tau (u,\overline{a})$ is finitely satisfiable in $\M$ then it is realized in $\M$.
%\end{defini}

For an extensive exposition on  models of $PA$ see \cite{k}. In this paper we investigate the valuation theoretical properties of
any $IPA - RCF$. Our main result Theorem  \ref{expgp} establishes that their value
groups (with respect to the natural valuation) satisfy strong valuation theoretical conditions (see Definitions \ref{lexp} and \ref{Exponential group}).
\begin{thm}
\label{expgp} Let $K$ be a non-Archimedean  $IPA-RCF$. Then $K$ admits left exponentiation and the
value group $G$  of $K$ is an exponential group in the additive
group $(\overline{K}, +, 0,<)$.
\end{thm}
Theorem
\ref{expgp} provides a method for constructing real closed fields
with no $IPA$, see Section \ref{section3}. Theorem \ref{expgp} has
many implications on the valuation theoretic structure of $IPA - RCF$.
In particular, non-Archimedean $IPA-RCF$s cannot be maximally valued,
see Corollary \ref{pownoipa}. Corollary \ref{pownoipa} can be
interpreted as an arithmetical counterpart to \cite[Theorem 1]{kks}
stating that no maximally valued non-Archimedean real closed field
admits left exponentiation. We remark that in  the above results
$PA$ can be replaced by its proper fragment $I\Delta_0+EXP$, see
Corollary \ref{leftexp}. The converse of Theorem  \ref{expgp} fails,
even in the countable case. Indeed in Section \ref{section4} we give
a characterization of countable recursively saturated exponential
groups - see Corollary \ref{countableexprecsatgp} - and use it to
construct countable exponential real closed fields which are not
$IPA-RCF$s, see Example \ref{monoexample}.

\smallskip

We conclude by the following question:  is every  model $Z$ of $PA$ the integer part of a real closed exponential field ($K$, exp)?
As we observe in Corollary \ref{shep} the natural candidate for $K$ is the real closure of the field of fractions of $Z$, which is an $IPA - RCF$.
By Theorem  \ref{expgp} we know that this real closed field admits a left exponential function.  However it is unknown to us whether every real closed field which admits a left exponential function also admits
a total exponential function. These questions will be investigated in future work.
%------------------------------------------------------------------------------------
\section{Valuation-theoretical Preliminaries}\label{section 2}

\label{preliminaries}

We first summarize some background on divisible ordered abelian groups (see \cite{K} and \cite{K2}). Let  $(G, +, 0, <)$ be a divisible  ordered  abelian group, i.e., an ordered $\mathbb{Q}$-vector space.  For any $x\in G$, let \mbox{$|x|=\max \{x,-x\}$.}  For non-zero $x, y \in G$, we define $x\sim y$ if there exists $n \in \N$ such that $n|x| \geq |y| $ and $ n|y| \geq |x|. $ We write $x\ll y$ if $n|x| < |y|$ for all $n \in \N$. Clearly,  $\sim$ is an equivalence relation, and $[x]$ denotes the equivalence class of any non-zero $x\in G$.
Let \mbox{$\Gamma : = \{[x] : x \in G\setminus\{ 0\} \}$.} We  define a total  order  $<$ on $\Gamma$  as follows $[y]\, <[x] $ if  $x\ll y$.
 We call $\Gamma$ the rank of $G$ (or the value set of $G$). The map
$v \colon G\  \to\ \Gamma \cup \{\infty\}$ defined by $v(0):= \infty$ and $v(x)= [x]$ if $x \neq 0$
is  called the natural valuation on $G$.  For every $\gamma \in \Gamma$, fix $x\in \gamma$ and choose a maximal Archimedean subgroup $B_{\gamma}$ of $G$ containing $x$.  We call $B_{\gamma}$  the  Archimedean component associated to $\gamma$. For each
$\gamma$, $B_{\gamma}$ is isomorphic to an ordered subgroup of $(\mathbb R, +,0,<).$

\smallskip

Now we recall analogous notions for real closed fields. Any  real
closed  field \mbox{$(K,+,\cdot ,0,1,<)$} has a natural valuation
$v$, that is the natural valuation associated with the divisible
ordered abelian  group $(K,+ ,0,<~)$. We define an addition on the
value set $v(K^\times)$ by setting \mbox{$v(x)+v(y)$} equal to
$v(xy)$ for all $x, y\in K^\times$.  The value set, equipped with
this addition, is a divisible ordered abelian group.    We call
$G=v(K^\times)$ the value group of $K$.

\smallskip
\noindent We now recall some basic notions and fix notations  relative to the natural valuation $ v$.

\noindent
{\bf Notations.}

\noindent $R_{v}:=\{ a\in K: v(a) \geq 0  \}$ is the valuation ring. Note that $(R_{v},+,0,<)$ is a divisible convex subgroup of $(K,+,0,<)$.

\noindent
$\mu_v:=\{ a\in K: v(a) > 0  \}$ the valuation ideal.

\noindent
$\mathcal{U}^{>0}_{v}:=\{ a\in K^{>0}: v(a) = 0  \}$ the group of positive units of $R_{v}$. Note that $(\mathcal{U}^{>0}_{v},\cdot , 1, <)$ is a divisible subgroup of $(K^{>0},\cdot , 1, <)$.

\noindent
$\overline{K}:=R_{v}/\mu_v$, the residue field of $K$,
which is an Archimedean field, i.e. isomorphic to a subfield of
$\mathbb R$.

\begin{rem}
\label{convexhull}
The valuation ring $R_v$ is the convex hull of $\mathbb Z$ in the ordered field $K$. It is straightforward to verify that $Z\cap R_v= \mathbb Z$ for any integer part $Z$ of $K$.
\end{rem}

\begin{defini}
\label{HahnGroup} (i) Let $\Gamma$ be a linearly ordered set and $\{
B_{\gamma}:\gamma \in \Gamma \}$ be a family of (additive)
Archimedean groups. The {\it Hahn group (Hahn sum)}
$G=\oplus_{\gamma \in \Gamma} B_{\gamma}$  is the group of functions
$f: \Gamma \rightarrow \cup_{\gamma \in \Gamma}B_{\gamma}$ where
$f(\gamma)\in B_{\gamma}$, $\mbox{supp } (f)=\{ \gamma \in \Gamma:
f(\gamma )\not=0 \}$ is finite, endowed with componentwise addition
and the lexicographical ordering.  Note that the rank of $G$ is
$\Gamma$ and the Archimedean components of $G$ are $\{
B_{\gamma}:\gamma \in \Gamma \}$.
\par
(ii) Let $G$ be an ordered abelian group and $k$ an Archimedean
field. Then $k((G))$ is the {\it field of generalized power series
over $G$ with  coefficients in $k$}. That is, $k((G))$ consists of
formal expressions of the form $\Sigma_{g\in G}c_{g}t^{g}$, where
$c_{g}\in k$ and $\{g\in G:c_{g}\neq 0\}$ is well-ordered. Again,
the addition is pointwise, the ordering is lexicographic and the
multiplication is given by the convolution formula.

\end{defini}

%-----------------------------------------------------------------------------------------------------------------------------------------------

\section{Left exponentiation on $IPA-RCF$ }\label{section3}

\noindent Let $(K,+,\cdot ,0,1,<)$ be a real closed field. Observe
that $(K,+,0,<)$ and $(K^{>0},\cdot,1,<)$ are divisible ordered
abelian groups, i.e. ordered  $\mathbb Q$-vector spaces. Therefore
we can   lexicographically decompose them as
$$(K,+,0,<)=  {A}\oplus R_{v}\mbox{ and } (K^{>0},\cdot,1,<) =  B\cdot \mathcal{U}_{v}^{>0},$$
where $A$ is a (vector space) complement of $R_{v}$ and $B$  is a (vector space) complement of $\mathcal{U}_{v}^{>0}.$

\medskip

\noindent We are now in a position to introduce the following

\begin{defini}
\label{lexp} A real closed field $K$ is said to {\it admit left
exponentiation} if there is an order preserving group isomorphism
from a complement ${A}$ of $R_{v}$ in $(K,+,0,<)$ onto a
complement $B$ of $\mathcal{U}_{v}^{>0}$ in $(K^{>0},\cdot,1,<)$.\\
\end{defini}

The existence of  left exponentiation on $K$  imposes very restrictive conditions on its value group.
These value groups were studied extensively in \cite{kks} and \cite{K}:

\begin{defini}
\label{Exponential group} Let $G$ be an ordered abelian group with
rank $\Gamma$ and Archimedean components $\{ B_{\gamma }: \gamma \in
\Gamma\}$. Let $C$ be an Archimedean  group. We say that $G$ is an
{\it exponential group in $C$} if $\Gamma$ is  isomorphic (as linear
order) to the negative cone $G^{<0}$, and each $B_{\gamma}$ is isomorphic
(as ordered group) to $C$.
\end{defini}
We shall need in Section \ref{section4} the following
characterization of countable divisible exponential groups
\cite[Theorem 1.42]{K}:
\begin{thm}\label{countableexpgp}
 A countable divisible ordered abelian group $G\not=\{0\}$ is an exponential
group
 if and only if $G$ is isomorphic to the Hahn sum $\displaystyle{
\oplus_{\mathbb{Q}}C}$ for some countable archimedean group $
C\not=\{0\}$.
\end{thm}
\noindent In the proof of Theorem \ref{expgp} we will make implicit
use of the general properties of the exponential function on a model
of $PA$ which we summarize here. The graph of the
exponential function over $\mathbb{N}$ is definable by an
$\mathcal{L}$-formula $E(x,y,z)$, which stands for $x^{y}=z$. For
convenience we fix the base $2$  for exponentiation. We use $E(x,y)$
for $E(2,x,y)$. $PA$ proves the following  properties: \noindent
\begin{enumerate}
 \item $E(0,1)$, $E(1,2)$ and $\forall{x}\exists!y E(x,y)$.
\item $\forall{x}\forall{y}((x\geq 1\wedge E(x,y))\rightarrow (y\geq x+1 ))$
\item $\forall{a,b,c,d}((E(a,c)\wedge E(b,d))\rightarrow E(a+b,cd))$
\item $\forall{a,b,c,d}((E(a,c)\wedge E(b,d)\wedge a<b)\rightarrow c<d)$
 ``the exponential is strictly increasing''
\item $\forall x\exists!y\exists z (E(y,z)\wedge z\leq x<2z)$
 ``each element lies between two successive powers of $2$''

\end{enumerate}
\noindent
The above properties imply that over any model of $PA$
there is a function exp defined by the formula $E(x,y)$ which
behaves as the exponential function over $\mathbb N$, and when
restricted to $\mathbb N$ is the usual exponential function in base
$2$.
%, i.e. satisfies the properties of the exponential function over $\mathbb N$.
In what follows we also write exp$(m)=2^m$.
%\models E(n,m)$ will be denoted either by $m=2^n$.
\begin{rem}
\label{fragment} It is known that there is a $\Delta_0$-formula
defining the graph of the exponential function in $\mathbb N$ for
which the above properties of exponentiation, except the totality,
are provable in $I\Delta_0$. These properties give an invariant
meaning to exponentiation in every non standard model of $I\Delta_0$
(see \cite[Chapter V.3]{hp}).  $EXP$ denotes the axiom
$\forall{x}\exists!y E(x,y)$.
\end{rem}

In Lemma 1.3 of \cite{MSS}, the authors show that if the non-Archimedean field $\mathbb{R}((G))$ is an $IPA-RCF$, then it would admit exponentiation. Theorem \ref{expgp} stated in the Introduction is a generalization of their result.
We are now ready to prove it.

%Let $(K,+,\cdot,0,1,<)$ be a real closed field and let $M$ be a nonstandard model of $PA$.
%Assume that $Z:=-M\cup M$ is an integer part of $K$. Then $K$ has left exponentiation.
\begin{proof}
We first show that $K$ admits left exponentiation. Let $Z$ be an
integer part of $K$ where $Z:=-M\cup M$ and $M$ is a model of $PA$.
The idea of the proof is to extend the exponential function defined
on $M$ first to $Z$ and then to the divisible hull $\mathbb
QZ:=\{qm:q\in\mathbb{Q}\wedge m\in Z\}$ of $Z$ using the rules of
exponentiation. Exploiting a direct sum decomposition of $\mathbb
QZ:=A\oplus \mathbb Q$, we get an exponential function exp on $A$ by
restriction. We will then show that $A$ is a complement to $R_v$ and
$B:=$exp$(A)$ a complement to $\mathcal U_v^{>0}$.

\medskip

\noindent Without loss of generality we fix a $\mathbb Q$-basis
$X\cup\{1\}\subset Z$ of $\mathbb QZ$ with $X\subset
Z\setminus{\mathbb Z}$. Set $A:=\langle X \rangle_{\mathbb Q}$, so
that $\mathbb QZ:=A\oplus \mathbb Q$. From Remark \ref{convexhull}
we infer that $\mathbb QZ\cap R_v=\mathbb Q$. Therefore for all
$0\not=a\in A$ we have $v(a)<0$.

\medskip

\noindent Extend  exp from $M$ to   $Z=-M\cup M$ (with values in the fraction field of $Z$)
in the obvious way by setting  exp$(-m):=\frac{1}{\mbox{exp}(m)}$
for $m\in M$.

\noindent Since $(K^{>0},\cdot,1,<)$ is divisible, we can define
exp$:\mathbb QZ\rightarrow K^{>0}$ by setting
$\mbox{exp}(\frac{m}{n}):=\mbox{exp}(m)^{\frac{1}{n}}$ for $m\in Z$,
$0<n\in\mathbb{N}$. Set $B:=\mbox{exp}(A)$.

\medskip

\noindent It is straightforward to verify that exp is strictly
increasing and a group homomorphism. Since $\mbox{exp}(n)=2^n$ for
$n\in \mathbb N$, we have by construction that $\mbox{exp}(z) \in
\mathbb Q$ for $z\in \mathbb Z$. It follows that for $q\in \mathbb
Q$ we have $\mbox{exp}(q)\in \mathcal{U}_{v}^{>0}$ since e.g.
$2^{\lfloor q \rfloor} \leq \mbox{exp}(q) < 2^{\lfloor q \rfloor +
1}$. We observe that if $0\not=a\in A$ then $v(a)<0$. It follows
that, for $a\in A$, if $a>0$ then $v(\mbox{exp}(a))<0$ and if $a<0$ then
$v(\mbox{exp}(a))>0$, so in any case $v(\mbox{exp}(a))\not=0$, so
$\mbox{exp}(a) \notin \mathcal{U}_{v}^{>0}$.

\medskip

\noindent \emph{Claim $1$}: $A$ is a complement of $R_{v}$ in $(K,+,0,<)$.

\noindent We first prove that $K=A+R_v$. Let $x\in K$. If $x\in
R_{v}$ there is nothing to prove, so assume that $x\in
K^{\times}\backslash R_{v}$.
 Write $\lfloor x\rfloor= \alpha + q\cdot 1$, with $\alpha \in A$ and $q\in\mathbb{Q}$. So $\lfloor x\rfloor\in A+R_{v}$. Since  $0\leq x-\lfloor x\rfloor<1$
 we have $x-\lfloor x\rfloor\in[0,1[\subset R_{v}$. Hence, $x=\lfloor x\rfloor+(x-\lfloor x\rfloor)\in A+R_{v}$.

\noindent Now let $0\not=x\in A\cap R_{v}$, then $v(x) <0$ and
$v(x)\geq 0$, a contradiction.

\medskip

\noindent \emph{Claim $2$}: $B$ is a complement of
$\mathcal{U}_{v}^{>0}$ in $(K^{>0},\cdot,1,<)$.

\noindent We first prove that $K^{>0}=B\cdot\mathcal{U}_{v}^{>0}$.
Let $x\in K^{>0}$. If $x\in \mathcal{U}_{v}^{>0}$, there is nothing
to prove, so assume that $x\notin \mathcal{U}_{v}^{>0}$. Without
loss of generality, let $x>1$, so $\lfloor x\rfloor \in M$. Since
$M\models PA$, by Axiom (4) there exists $m\in M$ such that
$\mbox{exp}(m)\leq \lfloor x\rfloor < \mbox{exp}(m+1)$. From these
inequalities it follows e.g. that $\mbox{exp}(m)\leq
x<2^{2}\mbox{exp}(m)$, so $1\leq \frac{x}{\mbox{exp}(m)}<4$. So,
$\frac{x}{\mbox{exp}(m)}\in \mathcal U_v^{>0}$. As $m\in M$, write
$m=\alpha+q$, where $\alpha \in A$, $q\in\mathbb{Q}$. Then
$\mbox{exp}(m)=\mbox{exp}(\alpha) \mbox{exp}(q)$. Now
$\mbox{exp}(q)\in \mathcal{U}_{v}^{>0}$, so
$\frac{x}{\mbox{exp}(\alpha)}\in \mathcal U_v^{>0}$ and $x\in B\cdot
\mathcal{U}_{v}^{>0}$.
%If $0<x<1$ then $1/x=b\cdot u$ for some $b=\mbox{exp}(a)\in B$, $u\in
%\mathcal{U}_{v}^{>0}$, and then $x=\mbox{exp}(-a)\cdot 1/u \in B\cdot
%\mathcal{U}_{v}^{>0}$ as required.
Finally, we verify that $\mathcal{U}_{v}^{>0}\cap B=\{1\}$. But this
is straightforward by the observation just before \emph{Claim $1$}.

\smallskip

\noindent In \cite[Proposition 1.22 and Corollary 1.23]{K},  it is
shown that for a non-Archimedean real closed field $K$, if $K$
admits left exponentiation, then the  value group of $K$ is an
exponential group in $C:=(\overline{K}, +, 0,<)$. This completes the
proof of the theorem.
\end{proof}

\begin{rem}
\label{noexpgp}

\noindent

1) In \cite[Corollary 2, p. 266]{Mend}, it is shown that $\mathbb{Z}$  cannot be a direct
summand of the additive group  of a non standard model of true
Arithmetic. The proof is easily seen to actually yield the same conclusion for models of $PA$.\footnote{In fact, it is shown in \cite{Ca} that this already holds
for models of the much weaker theory $IE_{2}$. Moreover, in an email to the first author, Emil Jerabek sketched an argument that this already holds for $IE_{1}$.}
So to define the required complement $A$ in the proof
of Theorem \ref{expgp} we rather use the fact that $\mathbb{Q}$ is a
direct summand of the divisible hull $\mathbb{Q}Z$ of $Z$.

\noindent

2) Note that in the above proof we may replace the hypothesis that
$K$ is real closed by the assumption that $(K^{>0}, \cdot)$ is
divisible, that is $K$ is root closed for positive elements.

\noindent

3) The converse of Theorem \ref{expgp} fails, even in the countable
case, see Section \ref{section4}, Example  \ref{monoexample}.

\noindent

4) The rank of a divisible exponential group $G\not= \{0\}$ is a
dense linear order without endpoints (since the linear order
$G^{<0}$ is such). There are plenty of divisible ordered abelian
groups that are not exponential groups in $(\overline{K}, +, 0,<)$.
For example, take the Hahn group $G=\oplus_{\gamma \in \Gamma}
B_{\gamma}$ where the Archimedean components $B_{\gamma}$ are
divisible but not all isomorphic and/or $\Gamma$ is not a dense
linear order without endpoints (say, a finite $\Gamma$).
Alternatively, we could choose all Archimedean components to be
divisible and all isomorphic, say to $C$, and $\Gamma$ to be a dense
linear order without endpoints, but choose the residue field so that
$(\overline{K},  +,0,<)$ is not isomorphic to $C$.
\end{rem}

\medskip

\noindent As an  application of Theorem \ref{expgp}, we present a
class of subfields  of power series fields that are not $IPA - RCF$.

\medskip

\noindent

\begin{example}\label{aclass}
[A class of not $IPA$ real closed fields:] Let $k$ be any real closed
field. Let $G\not=\{0\}$ be any divisible ordered
abelian group which is {\it not} an exponential group in
$(k,+,0,<)$, see Remark \ref{noexpgp} (4). Consider the field of power
series $k((G))$ and its subfield $k(G)$ generated by $k$ and $\{
t^g:g\in G\}$. Let $K$ be any real closed field satisfying
$$k(G)^{rc} \subseteq K \subseteq k((G))$$
where $k(G)^{rc}$ is the real closure of $k(G)$. Any such $K$ has
$G$ as value group and $k$ as residue field. By Theorem \ref{expgp},
$K$ is not an $IPA-RCF$.
\end{example}
\medskip

\noindent

The reader will also have noticed that the full power of $PA$ is by
far not used in the proof of Theorem \ref{expgp}. All that is needed
is a theory $T$ of arithmetic strong enough to prove arithmetical
facts (1)-(5) about the exponential function (see Remark
\ref{fragment}). Hence

\begin{corollary} \label{leftexp} Let $(R, +,\cdot , 0,1,<) $  be a non-archimedean  real closed field.
Assume that $R$ admits an integer part  which is a model of
$I\Delta_0+EXP$. Then $R$ admits a left exponential function.
\end{corollary}

Below we gather a few other consequences of Theorem \ref{expgp}. We
obtain the following exponential analogue to one direction of
Shepherdson's theorem  \cite{shepherdson} cited in the introduction.
For the notion of  exponential integer part appearing below and a proof of Ressayre's theorem that every exponential $RCF$ has an exponential $IP$, see
\cite{DKKL}.
\begin{corollary}\label{shep}
If $M\models PA$, then $M$ is an exponential integer part of a left
exponential real closed field.
\end{corollary}
\begin{proof}
Let $Z=-M\cup M$ and consider $K:=ff(Z)^{rc}$, the real closure of
the fraction field of $Z$. Then
$Z$ is an integer part of $K$ by the mentioned result of Shepherdson. As
remarked above since $M$ is a model of $PA$, it has a total
exponential function.
By Theorem \ref{expgp}, $K$ has left exponentiation.\\
\end{proof}
In \cite{kks} it is shown that for no ordered field $k$ and ordered abelian group $G\neq
\{0\}$ does the field $k((G))$ admit a left exponential
function. Therefore, if we take $k=\mathbb{R}$, we can deduce from Theorem \ref{expgp} the following result, without
any further assumptions on $G$.

\begin{corollary}
\label{pownoipa} For any divisible ordered abelian group $G\neq
\{0\}$ the field $\mathbb R((G))$ is not an $IPA- RCF$.
\end{corollary}

We also observe that the necessary condition on the value group
(appearing in Theorem \ref{expgp}) is not sufficient:

\begin{corollary}
\label{nosuff} Let $G$ be a divisible exponential group. Then there is a
real closed field $(K,+,\cdot,0,1,<)$ such that $v(K^{\times})=G$
but $K$ does not have an $IPA$.
\end{corollary}
\begin{proof}
Let $K=\mathbb R((G))$. By Corollary \ref{pownoipa}, $K$ cannot have
an $IPA$.
\end{proof}
Let $G$ be a divisible exponential group in $(\mathbb{R},+,0,<)$\footnote{For example, if $K$ is a non-Archimedean exponential field with residue field $\mathbb{R}$, then its value group $G$ satisfies this assumption, see \cite{K}.}. 
Applying Corollary \ref{nosuff}, we obtain in particular a real closed field $K$ with residue
field $\mathbb{R}$ such that its value group is exponential in $(\mathbb{R},+,0,<)$, but $K$ does not have an $IPA$.
%----------------------------------------------------------------------------------------------------------------------------------------
\section{Countable recursively saturated exponential
groups}\label{section4} We now investigate recursively saturated
real closed fields in light of Theorem \ref{expgp}.

\smallskip

 \noindent
For convenience we recall a few definitions, and refer the reader to
[DKL] for more details. $L$ be a computable language. An
$L$-structure $M$ is {\it recursively saturated} if for every
computable set of $L$-formulas $\tau(x, \bar{y})$ and  every tuple
$\bar{a}$ in $M$ (of the same length as $\bar{y}$) such that
$\tau(x,\bar{a})$ is finitely satisfiable in $M$,
$\tau(x,\bar{a})$ is realized in $M$. A subset $\mathcal{T}\subset
2^{<\omega}$ is a {\it tree} if every substring of an element of
$\mathcal{T}$ is also an element of $\mathcal{T}$. If $\sigma,
\tau\in 2^{<\omega}$, we let $\sigma\prec\tau$ denote that $\sigma$
is a substring of $\tau$. A sequence $f\in2^{\omega}$ is a {\it
path} through a tree $\mathcal{T}$ if for all $\sigma\in
2^{<\omega}$ with $\sigma\prec f$, we have $\sigma\in \mathcal{T}$.
A nonempty set $S\subset \mathbb{R}$ is a {\it Scott set} if $S$ is
{\it computably closed} i.e., if $r_1,\ldots r_n\in S$ and
$r\in\mathbb{R}$ is computable from $r_1\oplus\ldots \oplus r_n$
(the {\it Turing join} of $r_1,\ldots, r_n$), then $r\in S$; and
every infinite tree $\mathcal{T}\subset 2^{<\omega}$ which is
computable in some $r\in S$ has a path that is computable in some
$r'\in S$.

\medskip

\noindent

Any Scott set is an Archimedean real closed field, i.e. a real closed
subfield of $\mathbb{R}$, but there exist countable real closed
subfields of $\mathbb{R}$ which are not Scott sets such as the real closed field of recursive reals. For further relevant information on Scott sets,
we refer the reader to \cite{dklm}.

\medskip

\noindent

$IPA$-real closed fields are recursively saturated \footnote{In fact, it is shown in \cite[ Theorem 3.1]{jk} that a real closed field which is not recursively saturated does not even have
an IP which is a model of the much weaker theory $IE_{2}$ in which the set $\{x^{n}:n\in\mathbb{N}\}$ is bounded
for every element $x$.}, and the converse
holds in the countable case \cite[Theorem 5.1 and 5.2]{dks}. On the
other hand, an algebraic characterization of recursively saturated
real closed fields (of any cardinality) is given in \cite{DKL}. In
particular, it is shown that the value group of a recursively
saturated real closed field is recursively saturated \cite[Theorem
5.2]{DKL}. We will need the following characterization of countable
recursively saturated divisible ordered abelian groups
\cite[Corollary 4.6]{DKL}:
\begin{thm}\label{countablerecsatgp}
A countable divisible ordered abelian group $G\not=\{0\}$ is
recursively saturated if and only if $G$ is isomorphic to the Hahn
sum $\displaystyle{ \oplus_{\mathbb{Q}}S}$ for some countable Scott
set $S$.
\end{thm}

An immediate consequence of Theorems \ref{countableexpgp} and
\ref{countablerecsatgp} is
\begin{corollary} \label{countableexprecsatgp}
 A countable divisible exponential group $G\not=\{0\}$ in $C$ is
recursively saturated if and only if $C$ is a (countable) Scott set.
\end{corollary}
We now construct a countable non Archimedean left exponential real
closed field which is not recursively saturated, or equivalently,
which does not admit an $IPA$.

\begin{example}\label{monoexample}
We let $G:=\displaystyle{ \oplus_{\mathbb{Q}}(C, +)}$, where $C$ is
a countable real closed subfield of $\mathbb{R}$ which is {\it not}
a Scott set. We set $F:=C(t^g\>; g\in G)$ and $K:=F^{\rm rc}$ its
real closure. By \cite[Corollary 1.43]{K} $K$ admits a left
exponential function. On the other hand, by \cite[Theorem 5.2]{DKL}
and Corollary \ref{countableexprecsatgp} $K$ is not recursively
saturated.
\end{example}

\section{Acknowledgements}

The authors thank the anonymous referee for various helpful suggestions on the presentation.

\end{document}